\documentclass[12pt,reqno]{amsart}

\usepackage{amssymb}

\usepackage{enumitem}

\usepackage{graphicx}

\makeatletter
\@namedef{subjclassname@2010}{%
  \textup{2010} Mathematics Subject Classification}
\makeatother

\usepackage[T1]{fontenc}

\usepackage{xcolor}
\usepackage[colorlinks, linkcolor=blue, citecolor=blue, urlcolor=blue]{hyperref}

\newtheorem{theorem}{Theorem}[section]
\newtheorem{corollary}[theorem]{Corollary}
\newtheorem{fact}[theorem]{Fact}
\newtheorem{lemma}[theorem]{Lemma}
\newtheorem{proposition}[theorem]{Proposition}
\newtheorem{question}[theorem]{Question}
\newtheorem{claim}[theorem]{Claim}

\theoremstyle{definition}
\newtheorem{definition}[theorem]{Definition}

\DeclareMathOperator{\dom}{dom}
\DeclareMathOperator{\ran}{ran}
\DeclareMathOperator{\fin}{fin}
\DeclareMathOperator{\seq}{seq}
\DeclareMathOperator{\seqi}{seq^{1-1}}

\begin{document}


\baselineskip=17pt


\title[A choice-free cardinal equality]{A choice-free cardinal equality}

\author[Guozhen Shen]{Guozhen Shen}
\address{Institute of Mathematics\\
Academy of Mathematics and Systems Science\\
Chinese Academy of Sciences\\
Beijing 100190\\
People's Republic of China}
\address{School of Mathematical Sciences\\
University of Chinese Academy of Sciences\\
Beijing 100049\\
People's Republic of China}
\email{shen\_guozhen@outlook.com}

\date{}

\begin{abstract}
For a cardinal $\mathfrak{a}$, let $\fin(\mathfrak{a})$ be the cardinality
of the set of all finite subsets of a set which is of cardinality $\mathfrak{a}$.
It is proved without the aid of the axiom of choice that
for all infinite cardinals $\mathfrak{a}$ and all natural numbers $n$,
\[
2^{\fin(\mathfrak{a})^n}=2^{[\fin(\mathfrak{a})]^n}.
\]
On the other hand, it is proved that the following statement is consistent with $\mathsf{ZF}$:
there exists an infinite cardinal $\mathfrak{a}$ such that
\[
2^{\fin(\mathfrak{a})}<2^{\fin(\mathfrak{a})^2}<2^{\fin(\mathfrak{a})^3}<\dots<2^{\fin(\fin(\mathfrak{a}))}.
\]
\end{abstract}



\maketitle

\section{Introduction}
For a cardinal $\mathfrak{a}$, let $\fin(\mathfrak{a})$ be the cardinality
of the set of all finite subsets of a set which is of cardinality $\mathfrak{a}$.
The axiom of choice implies that $\fin(\mathfrak{a})=\mathfrak{a}$ for any infinite cardinal $\mathfrak{a}$.
However, in the absence of the axiom of choice, this is no longer the case.
In fact, in the ordered Mostowski model (cf.~\cite[pp.~198--202]{Halbeisen2017}),
the cardinality $\mathfrak{a}$ of the set of atoms satisfies
\begin{multline}\label{s001}
\fin(\mathfrak{a})<[\fin(\mathfrak{a})]^2<\fin(\mathfrak{a})^2<[\fin(\mathfrak{a})]^3<\fin(\mathfrak{a})^3<\cdots\\
<\fin(\fin(\mathfrak{a}))<\fin(\fin(\fin(\mathfrak{a})))<\dots<\aleph_0\cdot\fin(\mathfrak{a}).
\end{multline}
It is natural to ask which relationships between the \emph{powers} of
the cardinals in \eqref{s001} for an arbitrary infinite cardinal $\mathfrak{a}$
can be proved without the aid of the axiom of choice.

The first result of this kind is L\"auchli's lemma (cf.~\cite{Lauchli1961} or \cite[Lemma~5.27]{Halbeisen2017}),
which states that for all infinite cardinals $\mathfrak{a}$,
\[
2^{\aleph_0\cdot\fin(\mathfrak{a})}=2^{\fin(\mathfrak{a})}.
\]
L\"auchli's lemma implies that, in the ordered Mostowski model,
the powers of the cardinals in \eqref{s001} are all equal,
where $\mathfrak{a}$ is the cardinality of the set of atoms.

In this paper, we give a complete answer to the above question.
We first prove in $\mathsf{ZF}$ that for all infinite cardinals $\mathfrak{a}$,
\[
2^{\fin(\fin(\mathfrak{a}))}=2^{\fin(\fin(\fin(\mathfrak{a})))}=2^{\fin(\fin(\fin(\fin(\mathfrak{a}))))}=\cdots.
\]
Then, as our main result, we prove in $\mathsf{ZF}$ that
for all infinite cardinals $\mathfrak{a}$ and all natural numbers $n$,
\[
2^{\fin(\mathfrak{a})^n}=2^{[\fin(\mathfrak{a})]^n}.
\]
Finally, we prove that the following statement is consistent with $\mathsf{ZF}$:
there exists an infinite cardinal $\mathfrak{a}$ such that
\[
2^{\fin(\mathfrak{a})}<2^{\fin(\mathfrak{a})^2}<2^{\fin(\mathfrak{a})^3}<\dots<2^{\fin(\fin(\mathfrak{a}))}.
\]

\section{Basic notions and facts}
Throughout this paper, we shall work in $\mathsf{ZF}$.
In this section, we indicate briefly our use of some terminology and notation.
The cardinality of $x$, which we denote by $|x|$, is the least ordinal $\alpha$ equinumerous to $x$,
if $x$ is well-orderable, and the set of all sets $y$ of least rank which are equinumerous to $x$, otherwise.
We shall use lower case German letters $\mathfrak{a},\mathfrak{b}$ for cardinals.

For a function $f$, we shall use $\dom(f)$ for the domain of $f$, $\ran(f)$ for the range of $f$,
$f[x]$ for the image of $x$ under $f$, $f^{-1}[x]$ for the inverse image of $x$ under $f$,
and $f{\upharpoonright}x$ for the restriction of $f$ to $x$.
For functions $f$ and $g$, we use $g\circ f$ for the composition of $g$ and $f$.

\begin{definition}
Let $x,y$ be arbitrary sets, let $\mathfrak{a}=|x|$, and let $\mathfrak{b}=|y|$.
\begin{enumerate}[leftmargin=*, widest=1]
\item $x\preccurlyeq y$ means that there exists an injection from $x$ into $y$;
      $\mathfrak{a}\leqslant\mathfrak{b}$ means that $x\preccurlyeq y$.
\item $x\preccurlyeq^\ast y$ means that there exists a surjection from a subset of $y$ onto $x$;
      $\mathfrak{a}\leqslant^\ast\mathfrak{b}$ means that $x\preccurlyeq^\ast y$.
\item $\mathfrak{a}\nleqslant\mathfrak{b}$ ($\mathfrak{a}\nleqslant^\ast\mathfrak{b}$)
      denotes the negation of $\mathfrak{a}\leqslant\mathfrak{b}$ ($\mathfrak{a}\leqslant^\ast\mathfrak{b}$).
\item $\mathfrak{a}<\mathfrak{b}$ means that $\mathfrak{a}\leqslant\mathfrak{b}$ and $\mathfrak{b}\nleqslant\mathfrak{a}$.
\item $\mathfrak{a}=^\ast\mathfrak{b}$ means that $\mathfrak{a}\leqslant^\ast\mathfrak{b}$ and $\mathfrak{b}\leqslant^\ast\mathfrak{a}$.
\end{enumerate}
\end{definition}

It follows from the Cantor--Bernstein theorem that if $\mathfrak{a}\leqslant\mathfrak{b}$ and
$\mathfrak{b}\leqslant\mathfrak{a}$ then $\mathfrak{a}=\mathfrak{b}$.
Clearly, if $\mathfrak{a}\leqslant\mathfrak{b}$ then $\mathfrak{a}\leqslant^\ast\mathfrak{b}$,
and if $\mathfrak{a}\leqslant^\ast\mathfrak{b}$ then $2^\mathfrak{a}\leqslant2^\mathfrak{b}$.
Thus $\mathfrak{a}=^\ast\mathfrak{b}$ implies that $2^\mathfrak{a}=2^\mathfrak{b}$.

\begin{definition}
Let $x,y$ be arbitrary sets, let $\mathfrak{a}=|x|$, and let $\mathfrak{b}=|y|$.
\begin{enumerate}[leftmargin=*, widest=1]
\item $x^y$ is the set of all functions from $y$ into $x$; $\mathfrak{a}^\mathfrak{b}=|x^y|$.
\item $x^{\underline{y}}$ is the set of all injections from $y$ into $x$; $\mathfrak{a}^{\underline{\mathfrak{b}}}=|x^{\underline{y}}|$.
\item $[x]^y$ is the set of all subsets of $x$ which have the same cardinality as $y$; $[\mathfrak{a}]^\mathfrak{b}=|[x]^y|$.
\item $\seq(x)=\bigcup_{n\in\omega}x^n$; $\seq(\mathfrak{a})=|\seq(x)|$.
\item $\seqi(x)=\bigcup_{n\in\omega}x^{\underline{n}}$; $\seqi(\mathfrak{a})=|\seqi(x)|$.
\item $\fin(x)=\bigcup_{n\in\omega}[x]^n$; $\fin(\mathfrak{a})=|\fin(x)|$.
\end{enumerate}
\end{definition}

Below we list some basic properties of these cardinals.
We first note that $\fin(\mathfrak{a})\leqslant^\ast\seqi(\mathfrak{a})\leqslant\seq(\mathfrak{a})$.

\begin{fact}\label{s002}
For all cardinals $\mathfrak{a}$, $\seqi(\mathfrak{a})\leqslant\fin(\fin(\mathfrak{a}))$.
\end{fact}
\begin{proof}
For every set $x$, the function $f$ defined on $\seqi(x)$ given by $f(t)=\{t[n]\mid n\leqslant\dom(t)\}$
is an injection from $\seqi(x)$ into $\fin(\fin(x))$.
\end{proof}

\begin{lemma}\label{s003}
For all non-zero cardinals $\mathfrak{a}$, $\seq(\seq(\mathfrak{a}))=\seq(\mathfrak{a})$.
\end{lemma}
\begin{proof}
Cf.~\cite[Lemma~2]{Ellentuck1966}.
\end{proof}

\begin{lemma}\label{s004}
For all non-zero cardinals $\mathfrak{a}$,
$\seq(\mathfrak{a})=\aleph_0\cdot\seqi(\mathfrak{a})$.
\end{lemma}
\begin{proof}
Cf.~\cite[Lemma~2.22]{ShenYuan2019}.
\end{proof}

\begin{lemma}\label{s005}
For all infinite cardinals $\mathfrak{a}$, $\aleph_0\cdot\seqi(\mathfrak{a})\leqslant^\ast\seqi(\mathfrak{a})$.
\end{lemma}
\begin{proof}
Let $x$ be an infinite set. Let $p$ be a bijection from $\omega\times\omega$ onto $\omega$
such that $n\leqslant p(m,n)$ for any $m,n\in\omega$.
Let $f$ be the function defined on $\seqi(x)$ given by
\[
f(t)=(m,t{\upharpoonright}n),
\]
where $m,n\in\omega$ are such that $\dom(t)=p(m,n)$.
It is easy to see that $f$ is a surjection from $\seqi(x)$ onto $\omega\times\seqi(x)$.
\end{proof}

\begin{proposition}\label{s006}
For all infinite cardinals $\mathfrak{a}$,
\[
\seqi(\mathfrak{a})=^\ast\fin(\fin(\mathfrak{a}))=^\ast\fin(\fin(\fin(\mathfrak{a})))=^\ast\dots=^\ast\seq(\mathfrak{a}).
\]
\end{proposition}
\begin{proof}
Immediately follows from Fact~\ref{s002} and Lemmata~\ref{s003}, \ref{s004} and \ref{s005}.
\end{proof}

\begin{corollary}\label{s007}
For all infinite cardinals $\mathfrak{a}$,
\[
2^{\seqi(\mathfrak{a})}=2^{\fin(\fin(\mathfrak{a}))}=2^{\fin(\fin(\fin(\mathfrak{a})))}=\dots=2^{\seq(\mathfrak{a})}.
\]
\end{corollary}
\begin{proof}
Immediately follows from Proposition~\ref{s006}.
\end{proof}

The following lemma will be used in Section~\ref{s009}.

\begin{lemma}\label{s008}
For all cardinals $\mathfrak{a}$ and all $n\in\omega$, $\mathfrak{a}^{\underline{2^n}}\leqslant\fin(\mathfrak{a})^{n+1}$.
\end{lemma}
\begin{proof}
Let $x$ be an arbitrary set and let $n\in\omega$.
Let $f$ be the function defined on $x^{\underline{\wp(n)}}$ such that for all $t\in x^{\underline{\wp(n)}}$,
$f(t)$ is the function on $n+1$ given by
\[
f(t)(k)=
\begin{cases}
\{t(\varnothing)\},                       & \text{if $k=n$;}\\
\{t(a)\mid a\subseteq n\text{ and }k\in a\}, & \text{otherwise.}
\end{cases}
\]
Clearly, $\ran(f)\subseteq\fin(x)^{n+1}$.
It is easy to verify that for all $t\in x^{\underline{\wp(n)}}$,
$t$ is the function defined on $\wp(n)$ given by
\[
t(a)=
\begin{cases}
\bigcup f(t)(n),                                                                       & \text{if $a=\varnothing$;}\\
\bigcup\bigl(\bigcap_{k\in a}f(t)(k)\setminus\bigcup_{k\in n\setminus a}f(t)(k)\bigr), & \text{otherwise.}
\end{cases}
\]
Hence, $f$ is an injection from $x^{\underline{\wp(n)}}$ into $\fin(x)^{n+1}$.
\end{proof}

\section{The main theorem}
In this section, we prove our main result which states that
for all infinite cardinals $\mathfrak{a}$ and all natural numbers $n$,
\[
2^{\fin(\mathfrak{a})^n}=2^{[\fin(\mathfrak{a})]^n}.
\]
The main idea of the proof is originally from \cite{Lauchli1961}.

Fix an arbitrary infinite set $A$ and a non-zero natural number $n$.
For a finite sequence $\langle x_1,\dots,x_n\rangle$ of length $n$,
we write $\vec{x}=\langle x_1,\dots,x_n\rangle$ for short.
For finite sequences $\vec{x}=\langle x_1,\dots,x_n\rangle$ and $\vec{y}=\langle y_1,\dots,y_n\rangle$,
we introduce the following abbreviations:
$\vec{x}\sqsubseteq\vec{y}$ means that $x_i\subseteq y_i$ for any $i=1,\dots,n$;
$\vec{x}\sqsubset\vec{y}$ means that $\vec{x}\sqsubseteq\vec{y}$ but $\vec{x}\neq\vec{y}$;
$\vec{x}\sqcup\vec{y}$ denotes the finite sequence $\langle x_1\cup y_1,\dots,x_n\cup y_n\rangle$;
$\vec{x}\sqcap\vec{y}$ denotes the finite sequence $\langle x_1\cap y_1,\dots,x_n\cap y_n\rangle$;
$\vec{\varnothing}$ denotes the finite sequence $\langle\varnothing,\dots,\varnothing\rangle$ of length $n$.
For an operator $H$ and an $m\in\omega$, we write $H^{(m)}(X)$ for $H(H(\cdots H(X)\cdots))$ ($m$ times),
and if $m=0$ then $H^{(0)}(X)$ is $X$ itself.

\begin{definition}
For all natural numbers $k_1,\dots,k_n$ and $l_1,\dots,l_n$ such that $k_i\leqslant l_i$ for any $i=1,\dots,n$,
we introduce the following three functions:
\begin{enumerate}[leftmargin=*, widest=1]
\item $F_{n,\vec{k},\vec{l}\,}$ is the function defined on $\wp([A]^{k_1}\times\dots\times[A]^{k_n})$ given by
\[
F_{n,\vec{k},\vec{l}\,}(X)=\bigl\{\vec{y}\in[A]^{l_1}\times\dots\times[A]^{l_n}\!\bigm|\vec{x}\sqsubseteq\vec{y}\text{ for some }\vec{x}\in X\bigr\};
\]
\item $G_{n,\vec{k},\vec{l}\,}$ is the function defined on $\wp([A]^{k_1}\times\dots\times[A]^{k_n})$ given by
\[
G_{n,\vec{k},\vec{l}\,}(X)=\left\{\vec{x}\in[A]^{k_1}\times\dots\times[A]^{k_n}\,\middle|
\begin{array}{l}
\text{for all }\vec{y}\in[A]^{l_1}\times\dots\times[A]^{l_n}\!\!\!\\
\text{if }\vec{x}\sqsubseteq\vec{y}\text{ then }\vec{y}\in F_{n,\vec{k},\vec{l}\,}(X)
\end{array}
\right\};
\]
\item $H_{n,\vec{k},\vec{l}\,}$ is the function defined on $\wp([A]^{k_1}\times\dots\times[A]^{k_n})$ given by
\[
H_{n,\vec{k},\vec{l}\,}(X)=G_{n,\vec{k},\vec{l}\,}(X)\setminus X.
\]
\end{enumerate}
\end{definition}

The proof of the following fact is easy and will be omitted.

\begin{fact}\label{s011}
Let $k_1,\dots,k_n$ and $l_1,\dots,l_n$ be natural numbers such that $k_i\leqslant l_i$ for any $i=1,\dots,n$.
\begin{enumerate}[label=\upshape(\roman*), leftmargin=*, widest=xiii]
\item If $X\subseteq Y\subseteq[A]^{k_1}\times\dots\times[A]^{k_n}$ then $F_{n,\vec{k},\vec{l}\,}(X)\subseteq F_{n,\vec{k},\vec{l}\,}(Y)$.\label{s011a}
\item If $X\subseteq[A]^{k_1}\times\dots\times[A]^{k_n}$ then $X\subseteq G_{n,\vec{k},\vec{l}\,}(X)$.\label{s011b}
\item If $X\subseteq Y\subseteq[A]^{k_1}\times\dots\times[A]^{k_n}$ then $G_{n,\vec{k},\vec{l}\,}(X)\subseteq G_{n,\vec{k},\vec{l}\,}(Y)$.\label{s011c}
\item If $X\subseteq[A]^{k_1}\times\dots\times[A]^{k_n}$
      then $G_{n,\vec{k},\vec{l}\,}(G_{n,\vec{k},\vec{l}\,}(X))=G_{n,\vec{k},\vec{l}\,}(X)$.\label{s011d}
\item If $X\subseteq[A]^{k_1}\times\dots\times[A]^{k_n}$
      then $F_{n,\vec{k},\vec{l}\,}(G_{n,\vec{k},\vec{l}\,}(X))=F_{n,\vec{k},\vec{l}\,}(X)$.\label{s011e}
\item $F_{n,\vec{k},\vec{l}\,}$ is injective on $\{X\subseteq[A]^{k_1}\times\dots\times[A]^{k_n}\mid G_{n,\vec{k},\vec{l}\,}(X)=X\}$.\label{s011f}
\item If $X\subseteq[A]^{k_1}\times\dots\times[A]^{k_n}$ and $m\in\omega$ then \label{s011g}
\[
H_{n,\vec{k},\vec{l}\,}^{(m)}(X)=G_{n,\vec{k},\vec{l}\,}(H_{n,\vec{k},\vec{l}\,}^{(m)}(X))\setminus H_{n,\vec{k},\vec{l}\,}^{(m+1)}(X).
\]
\item Let $l'_1,\dots,l'_n$ be natural numbers such that $l_i\leqslant l'_i$ for any $i=1,\dots,n$.
      If $X\subseteq[A]^{k_1}\times\dots\times[A]^{k_n}$ then $G_{n,\vec{k},\vec{l}\,}(X)\subseteq G_{n,\vec{k},\vec{l'}}(X)$,
      and hence $G_{n,\vec{k},\vec{l'}}(X)=X$ implies that $G_{n,\vec{k},\vec{l}\,}(X)=X$.\label{s011h}
\end{enumerate}
\end{fact}

The key step of our proof is the following lemma.

\begin{lemma}\label{s010}
For all natural numbers $k_1,\dots,k_n$ and $l_1,\dots,l_n$ such that $k_i\leqslant l_i$ for any $i=1,\dots,n$,
if $X\subseteq[A]^{k_1}\times\dots\times[A]^{k_n}$ then
\[
H_{n,\vec{k},\vec{l}\,}^{(k_1+\dots+k_n+1)}(X)=\varnothing.
\]
\end{lemma}

Before we prove Lemma~\ref{s010}, we use it to prove our main theorem.

\begin{theorem}\label{s012}
For all infinite cardinals $\mathfrak{a}$ and all natural numbers $n$,
\[
2^{\fin(\mathfrak{a})^n}=2^{[\fin(\mathfrak{a})]^n}.
\]
\end{theorem}
\begin{proof}
Let $A$ be an infinite set such that $|A|=\mathfrak{a}$.
The case $n=0$ is obvious. So assume that $n$ is a non-zero natural number.
For all natural numbers $k_1,\dots,k_n,m$, let $s(\vec{k},m)$ be the finite sequence
\[
\langle p_1^{k_1}\cdots p_n^{k_n}p_{n+1}^mp_{n+2}^i\rangle_{1\leqslant i\leqslant n}
\]
where $p_j$ is the $j$-th prime number, and let $t(\vec{k})=s(\vec{k},k_1+\dots+k_n)$.

For all $X\subseteq\fin(A)^n$ and all natural numbers $k_1,\dots,k_n,m$, we define
\begin{align*}
X_{\vec{k}}   & =X\cap([A]^{k_1}\times\dots\times[A]^{k_n});\\
Y_{\vec{k},m} & =G_{n,\vec{k},t(\vec{k})}(H_{n,\vec{k},t(\vec{k})}^{(m)}(X_{\vec{k}}));\\
Z_{\vec{k},m} & =F_{n,\vec{k},s(\vec{k},m)}(Y_{\vec{k},m}).
\end{align*}
Notice that for any finite sequence $\vec{x}=\langle x_1,\dots,x_n\rangle$, $\ran(\vec{x})=\{x_1,\dots,x_n\}$.
Now, let $\Phi$ be the function defined on $\wp(\fin(A)^n)$ given by
\[
\Phi(X)=\bigl\{\ran(\vec{y})\bigm|\exists k_1,\dots,k_n,m\in\omega\,\bigl(m\leqslant k_1+\dots+k_n\text{ and }\vec{y}\in Z_{\vec{k},m}\bigr)\bigr\}.
\]
We claim that $\Phi$ is an injection from $\wp(\fin(A)^n)$ into $\wp([\fin(A)]^n)$.

Let $X\subseteq\fin(A)^n$. For all $\vec{y}=\langle y_1,\dots,y_n\rangle\in Z_{\vec{k},m}$,
it is easy to see that $|y_i|=p_1^{k_1}\cdots p_n^{k_n}p_{n+1}^mp_{n+2}^i$ for any $i=1,\dots,n$,
and thus $|y_1|<\dots<|y_n|$, which implies that $\ran(\vec{y})\in[\fin(A)]^n$.
Hence $\Phi(X)\subseteq[\fin(A)]^n$. Moreover, $X$ is uniquely determined by $\Phi(X)$ in the following way:

First, for all natural numbers $k_1,\dots,k_n,m$ such that $m\leqslant k_1+\dots+k_n$,
$Z_{\vec{k},m}$ is uniquely determined by $\Phi(X)$:
\[
Z_{\vec{k},m}=\bigl\{\vec{y}\in[A]^{l_1}\times\dots\times[A]^{l_n}\bigm|\ran(\vec{y})\in\Phi(X)\bigr\},
\]
where $l_i=p_1^{k_1}\cdots p_n^{k_n}p_{n+1}^mp_{n+2}^i$ for any $i=1,\dots,n$.

Then, for all natural numbers $k_1,\dots,k_n,m$ such that $m\leqslant k_1+\dots+k_n$,
by Fact~\ref{s011}\ref{s011d}\ref{s011f}\ref{s011h},
$Y_{\vec{k},m}$ is the unique subset of $[A]^{k_1}\times\dots\times[A]^{k_n}$ such that
$G_{n,\vec{k},t(\vec{k})\,}(Y_{\vec{k},m})=Y_{\vec{k},m}$ and $F_{n,\vec{k},s(\vec{k},m)}(Y_{\vec{k},m})=Z_{\vec{k},m}$,
which implies that $Y_{\vec{k},m}$ is uniquely determined by $\Phi(X)$.

Now, for all natural numbers $k_1,\dots,k_n$, it follows from Fact~\ref{s011}\ref{s011g} and Lemma~\ref{s010} that
\[
X_{\vec{k}}=Y_{\vec{k},0}\setminus(Y_{\vec{k},1}\setminus(\cdots(Y_{\vec{k},k_1+\dots+k_n-1}\setminus Y_{\vec{k},k_1+\dots+k_n})\cdots)),
\]
and thus $X_{\vec{k}}$ is uniquely determined by $\Phi(X)$.

Finally, since
\[
X=\bigcup_{k_1,\dots,k_n\in\omega}X_{\vec{k}},
\]
it follows that $X$ is also uniquely determined by $\Phi(X)$.

Hence, $\Phi$ is an injection from $\wp(\fin(A)^n)$ into $\wp([\fin(A)]^n)$,
and thus $2^{\fin(\mathfrak{a})^n}\leqslant2^{[\fin(\mathfrak{a})]^n}$.
Since $[\fin(\mathfrak{a})]^n\leqslant^\ast\fin(\mathfrak{a})^n$,
it follows that $2^{[\fin(\mathfrak{a})]^n}\leqslant2^{\fin(\mathfrak{a})^n}$,
and thus $2^{\fin(\mathfrak{a})^n}=2^{[\fin(\mathfrak{a})]^n}$ follows from the Cantor--Bernstein theorem.
\end{proof}

We still have to prove Lemma~\ref{s010}.
To this end, we need the following version of Ramsey's theorem, whose proof will be omitted.

\begin{lemma}\label{s013}
Let $n$ be a non-zero natural number. There exists a function $R$
defined on $\omega^n\times(\omega\setminus\{0\})\times\omega$ such that
for all natural numbers $j_1,\dots,j_n,c,r$ with $c>0$ and all finite sets $S_1,\dots,S_n,Y_1,\dots,Y_c$,
if $|S_i|\geqslant R(j_1,\dots,j_n,c,r)$ for any $i=1,\dots,n$ and
\[
[S_1]^{j_1}\times\dots\times[S_n]^{j_n}=Y_1\cup\dots\cup Y_c,
\]
then for each $i=1,\dots,n$ there exist a $T_i\in[S_i]^r$ such that
\[
[T_1]^{j_1}\times\dots\times[T_n]^{j_n}\subseteq Y_d
\]
for some $d=1,\dots,c$.
\end{lemma}

\begin{proof}[Proof of Lemma~\textnormal{\ref{s010}}]
Let $A$ be an arbitrary infinite set and $n$ a non-zero natural number.
Let $k_1,\dots,k_n$ and $l_1,\dots,l_n$ be natural numbers such that $k_i\leqslant l_i$ for any $i=1,\dots,n$.
Since in this proof the natural numbers $n,k_1,\dots,k_n,l_1,\dots,l_n$ are fixed,
we shall omit the subscripts in $F_{n,\vec{k},\vec{l}\,}$,
$G_{n,\vec{k},\vec{l}\,}$ and $H_{n,\vec{k},\vec{l}\,}$ for convenience.

\pagebreak

Consider the following two formulae:
\begin{description}
  \item[$\phi(X,\vec{x},\vec{y})$] $X\subseteq[A]^{k_1}\times\dots\times[A]^{k_n}$ and
        $\vec{x},\vec{y}\in\fin(A)^n$ are such that $|x_i|\leqslant k_i$ for any $i=1,\dots,n$,
        such that $\vec{x}\sqcap\vec{y}=\vec{\varnothing}$,
        and such that $\vec{x}\sqcup\vec{z}\in X$ for any $\vec{z}\in[y_1]^{k_1-|x_1|}\times\dots\times[y_n]^{k_n-|x_n|}$.
  \item[$\psi(X,\vec{x})$] For all $r\in\omega$ there exists a $\vec{y}\in([A]^r)^n$ such that $\phi(X,\vec{x},\vec{y})$.
\end{description}

We claim that for all $X\subseteq[A]^{k_1}\times\dots\times[A]^{k_n}$ and all $\vec{x}\in\fin(A)^n$,
\begin{equation}\label{s014}
\text{if $\psi(H(X),\vec{x})$ then $\psi(X,\vec{u})$ for some $\vec{u}\sqsubset\vec{x}$.}
\end{equation}
Once we prove \eqref{s014}, we finish the proof of Lemma~\ref{s010} as follows.
Assume towards a contradiction that $X\subseteq[A]^{k_1}\times\dots\times[A]^{k_n}$
and there exists an $\vec{x}\in H^{(k_1+\dots+k_n+1)}(X)$.
It is obvious that $\psi(H^{(k_1+\dots+k_n+1)}(X),\vec{x})$.
Now, by repeatedly applying \eqref{s014}, we get a descending sequence
\[
\vec{x}\sqsupset\vec{u}_1\sqsupset\dots\sqsupset\vec{u}_{k_1+\dots+k_n+1},
\]
which is absurd, since $\vec{x}\in[A]^{k_1}\times\dots\times[A]^{k_n}$.

Now, let us prove \eqref{s014}. Let $X\subseteq[A]^{k_1}\times\dots\times[A]^{k_n}$
and let $\vec{x}\in\fin(A)^n$ be such that $\psi(H(X),\vec{x})$. It suffices to prove that
\begin{equation}\label{s015}
\forall r\geqslant l_1+\dots+l_n\,\exists\vec{u}\sqsubset\vec{x}\,\exists\vec{y}\in([A]^r)^n\,\phi(X,\vec{u},\vec{y}),
\end{equation}
since then there must be a $\vec{u}\sqsubset\vec{x}$ such that for infinitely many $r\in\omega$
there exists a $\vec{y}\in([A]^r)^n$ such that $\phi(X,\vec{u},\vec{y})$,
and for this $\vec{u}$ we have $\psi(X,\vec{u})$.

We prove \eqref{s015} as follows. Let $r\geqslant l_1+\dots+l_n$.
Let $R$ be the function whose existence is asserted by Lemma~\ref{s013}. We define
\begin{align*}
r'  & =\max\{R(j_1,\dots,j_n,2,r)\mid j_i\leqslant k_i\text{ for any }i=1,\dots,n\};\\
r'' & =R(l_1-|x_1|,\dots,l_n-|x_n|,2^{|x_1|+\dots+|x_n|},r').
\end{align*}
Since $\psi(H(X),\vec{x})$, we can find an $\vec{S}=\langle S_1,\dots,S_n\rangle\in([A]^{r''})^n$ such that $\phi(H(X),\vec{x},\vec{S})$.
Notice that $\vec{x}\sqcap\vec{S}=\vec{\varnothing}$. For each $\vec{u}\sqsubseteq\vec{x}$, let
\[
Y_{\vec{u}}=\bigl\{\vec{w}\in[S_1]^{l_1-|x_1|}\times\dots\times[S_n]^{l_n-|x_n|}\bigm|
\vec{u}\sqcup\vec{v}\in X\text{ for some }\vec{v}\sqsubseteq\vec{w}\bigr\}.
\]

We claim that
\begin{equation}\label{s016}
[S_1]^{l_1-|x_1|}\times\dots\times[S_n]^{l_n-|x_n|}=\textstyle\bigcup\{Y_{\vec{u}}\mid\vec{u}\sqsubseteq\vec{x}\}.
\end{equation}
Let $\vec{w}\in[S_1]^{l_1-|x_1|}\times\dots\times[S_n]^{l_n-|x_n|}$.
Take a $\vec{z}\in[S_1]^{k_1-|x_1|}\times\dots\times[S_n]^{k_n-|x_n|}$ such that $\vec{z}\sqsubseteq\vec{w}$.
Then it follows from $\phi(H(X),\vec{x},\vec{S})$ that $\vec{x}\sqcup\vec{z}\in H(X)$, and thus $\vec{x}\sqcup\vec{z}\in G(X)$.
Since $\vec{x}\sqcup\vec{z}\sqsubseteq\vec{x}\sqcup\vec{w}\in[A]^{l_1}\times\dots\times[A]^{l_n}$,
it follows that $\vec{x}\sqcup\vec{w}\in F(X)$, and hence $\vec{a}\sqsubseteq\vec{x}\sqcup\vec{w}$ for some $\vec{a}\in X$.
Now, if we take $\vec{u}=\vec{a}\sqcap\vec{x}$ and $\vec{v}=\vec{a}\sqcap\vec{w}$,
then we have $\vec{u}\sqcup\vec{v}=\vec{a}\in X$ and hence $\vec{w}\in Y_{\vec{u}}$.

By \eqref{s016} and Lemma~\ref{s013}, we can find a $\vec{u}=\langle u_1,\dots,u_n\rangle\sqsubseteq\vec{x}$
such that for each $i=1,\dots,n$ there exist a $T_i\in[S_i]^{r'}$ such that
\begin{equation}\label{s019}
[T_1]^{l_1-|x_1|}\times\dots\times[T_n]^{l_n-|x_n|}\subseteq Y_{\vec{u}}.
\end{equation}
Let
\[
Z=\bigl\{\vec{v}\in[T_1]^{k_1-|u_1|}\times\dots\times[T_n]^{k_n-|u_n|}\bigm|\vec{u}\sqcup\vec{v}\in X\bigr\}.
\]
Since $|T_i|=r'\geqslant R(k_1-|u_1|,\dots,k_n-|u_n|,2,r)$ for any $i=1,\dots,n$,
it follows from Lemma~\ref{s013} that we can find a $\vec{y}=\langle y_1,\dots,y_n\rangle$
such that $y_i\in[T_i]^r$ for any $i=1,\dots,n$, and such that either
\begin{equation}\label{s017}
[y_1]^{k_1-|u_1|}\times\dots\times[y_n]^{k_n-|u_n|}\subseteq Z
\end{equation}
or
\begin{equation}\label{s018}
([y_1]^{k_1-|u_1|}\times\dots\times[y_n]^{k_n-|u_n|})\cap Z=\varnothing.
\end{equation}

We claim that \eqref{s018} is impossible. Since $|y_i|=r\geqslant l_i\geqslant l_i-|x_i|$ for any $i=1,\dots,n$,
there is a $\vec{w}\in[y_1]^{l_1-|x_1|}\times\dots\times[y_n]^{l_n-|x_n|}$,
and thus it follows from \eqref{s019} that $\vec{w}\in Y_{\vec{u}}$,
which implies that $\vec{u}\sqcup\vec{v}\in X$ for some $\vec{v}\sqsubseteq\vec{w}$
and such a $\vec{v}$ is in $([y_1]^{k_1-|u_1|}\times\dots\times[y_n]^{k_n-|u_n|})\cap Z$.
Therefore \eqref{s017} must hold, from which $\phi(X,\vec{u},\vec{y})$ follows.

It remains to show that $\vec{u}\neq\vec{x}$.
Since $\phi(H(X),\vec{x},\vec{S})$ and $\vec{y}\sqsubseteq\vec{S}$, it follows that $\phi(H(X),\vec{x},\vec{y})$.
If $\vec{u}=\vec{x}$, then we also have $\phi(X,\vec{x},\vec{y})$, which is impossible:
Since $|y_i|=r\geqslant l_i\geqslant k_i\geqslant k_i-|x_i|$ for any $i=1,\dots,n$,
there is a $\vec{z}\in[y_1]^{k_1-|x_1|}\times\dots\times[y_n]^{k_n-|x_n|}$,
and for such a $\vec{z}$, we cannot have both $\vec{x}\sqcup\vec{z}\in H(X)$ and $\vec{x}\sqcup\vec{z}\in X$.
\end{proof}

\section{Consistency results}\label{s009}
In this section, we establish some consistency results by the method of permutation models.
Permutation models are not models of $\mathsf{ZF}$;
they are models of $\mathsf{ZFA}$ (the Zermelo-Fraenkel set theory with atoms).
Nevertheless, they indirectly give, via the Jech--Sochor theorem (cf.~\cite[Theorem~17.2]{Halbeisen2017}),
models of $\mathsf{ZF}$.

For our purpose, we only consider the basic Fraenkel model $\mathcal{V}_\mathrm{F}$ (cf.~\cite[pp.~195--196]{Halbeisen2017}).
The set $A$ of atoms of $\mathcal{V}_\mathrm{F}$ is denumerable,
and $x\in\mathcal{V}_\mathrm{F}$ if and only if $x\subseteq\mathcal{V}_\mathrm{F}$
and $x$ has a \emph{finite support}, that is, a set $B\in\fin(A)$ such that
every permutation of $A$ fixing $B$ pointwise also fixes $x$.

\begin{lemma}\label{s020}
Let $A$ be the set of atoms of $\mathcal{V}_\mathrm{F}$ and let $\mathfrak{a}=|A|$.
In $\mathcal{V}_\mathrm{F}$,
\[
2^{\fin(\mathfrak{a})}<2^{\fin(\mathfrak{a})^2}<2^{\fin(\mathfrak{a})^3}<\dots<2^{\fin(\fin(\mathfrak{a}))}.
\]
\end{lemma}
\begin{proof}
Let $n\in\omega$. We claim that in $\mathcal{V}_\mathrm{F}$,
\begin{equation}\label{s021}
2^{\mathfrak{a}^{\underline{2^n}}}\nleqslant2^{\fin(\mathfrak{a})^n}.
\end{equation}
Assume towards a contradiction that there exists an injection $f\in\mathcal{V}_\mathrm{F}$
from $\wp(A^{\underline{2^n}})$ into $\wp(\fin(A)^n)$. Let $B$ be a finite support of $f$.
Take an arbitrary $C\in[A\setminus B]^{2^n+1}$ and a $u\in C^{\underline{2^n}}$.
We say that a permutation $\pi$ of $A$ is \emph{even} (\emph{odd})
if $\pi$ moves only elements of $C$ and can be written as a product of an even (odd) number of transpositions.
It is well-known that a permutation of $A$ cannot be both even and odd. Now, let
\[
\mathcal{E}=\{\pi(u)\mid\pi\text{ is an even permutation of }A\},
\]
and let
\[
\mathcal{O}=\{\sigma(u)\mid\sigma\text{ is an odd permutation of }A\}.
\]
Clearly, $\{\mathcal{E},\mathcal{O}\}$ is a partition of $C^{\underline{2^n}}$,
for all even permutations $\pi$ of $A$ we have $\pi(\mathcal{E})=\mathcal{E}$,
and for all odd permutations $\sigma$ of $A$ we have $\sigma(\mathcal{E})=\mathcal{O}$.
Now, let us consider $f(\mathcal{E})$.
For each $t\in f(\mathcal{E})$, let $\sim_t$ be the equivalence relation on $C$ such that for all $a,b\in C$,
\[
a\sim_tb\quad\text{if and only if}\quad\forall k<n\,\bigl(a\in t(k)\leftrightarrow b\in t(k)\bigr).
\]
For all even permutations $\pi$ of $A$, since $B$ is a finite support of $f$,
it follows that $\pi(f)=f$, and thus $\pi(f(\mathcal{E}))=f(\mathcal{E})$.
For all odd permutations $\sigma$ of $A$ and all $t\in f(\mathcal{E})$,
since $|C/{\sim_t}|\leqslant2^n$ and $|C|=2^n+1$, there are $a,b\in C$ such that $a\neq b$ and $a\sim_tb$,
and therefore the transposition $\tau$ that swaps $a$ and $b$ fixes $t$,
which implies that $\sigma(t)=(\sigma\circ\tau)(t)\in f(\mathcal{E})$ since $\sigma\circ\tau$ is even.
Hence, for all odd permutations $\sigma$ of $A$, $\sigma(f(\mathcal{E}))=f(\mathcal{E})$,
which implies that $f(\mathcal{O})=f(\sigma(\mathcal{E}))=\sigma(f(\mathcal{E}))=f(\mathcal{E})$,
contradicting the injectivity of $f$.

Now, it follows from Lemma~\ref{s008} that $\mathfrak{a}^{\underline{2^n}}\leqslant\fin(\mathfrak{a})^{n+1}$,
and therefore $2^{\mathfrak{a}^{\underline{2^n}}}\leqslant2^{\fin(\mathfrak{a})^{n+1}}$,
which implies that $2^{\fin(\mathfrak{a})^n}<2^{\fin(\mathfrak{a})^{n+1}}$ by \eqref{s021}.
It follows from Theorem~\ref{s012} that $2^{\fin(\mathfrak{a})^n}=2^{[\fin(\mathfrak{a})]^n}\leqslant 2^{\fin(\fin(\mathfrak{a}))}$. Hence
\[
2^{\fin(\mathfrak{a})}<2^{\fin(\mathfrak{a})^2}<2^{\fin(\mathfrak{a})^3}<\dots<2^{\fin(\fin(\mathfrak{a}))}.\qedhere
\]
\end{proof}

Now the following proposition immediately follows from Lemma~\ref{s020} and the Jech--Sochor theorem.

\begin{proposition}\label{s022}
The following statement is consistent with $\mathsf{ZF}$:
there is an infinite cardinal $\mathfrak{a}$ such that
\[
2^{\fin(\mathfrak{a})}<2^{\fin(\mathfrak{a})^2}<2^{\fin(\mathfrak{a})^3}<\dots<2^{\fin(\fin(\mathfrak{a}))}.
\]
\end{proposition}

It is natural to wonder whether the conclusion of Theorem~\ref{s012}
can be strengthened to $\fin(\mathfrak{a})^n\leqslant^\ast[\fin(\mathfrak{a})]^n$.
We shall give a negative answer to this question.
The case $n=1$ of the following lemma is proved in \cite{Truss1974}.

\begin{lemma}\label{s023}
Let $A$ be the set of atoms of $\mathcal{V}_\mathrm{F}$.
In $\mathcal{V}_\mathrm{F}$, for every $n\in\omega$, $\fin(A)^n$ is dually Dedekind finite;
that is, every surjection from $\fin(A)^n$ onto $\fin(A)^n$ is injective.
\end{lemma}
\begin{proof}
Let $n\in\omega$. Take an arbitrary surjection $f\in\mathcal{V}_\mathrm{F}$ from $\fin(A)^n$ onto $\fin(A)^n$.
In order to prove the injectivity of $f$, it suffices to show that
\begin{equation}\label{s024}
\text{for all $t\in\fin(A)^n$ there is an $m>0$ such that $f^{(m)}(t)=t$.}
\end{equation}
Let $B$ be a finite support of $f$. For each $t\in\fin(A)^n$,
let $\sim_t$ be the equivalence relation on $A\setminus B$ such that for all $a,b\in A\setminus B$,
\[
a\sim_tb\quad\text{if and only if}\quad\forall k<n\,\bigl(a\in t(k)\leftrightarrow b\in t(k)\bigr).
\]
Let $\sqsubseteq$ be the preorder on $\fin(A)^n$, such that for all $t,u\in\fin(A)^n$,
\[
t\sqsubseteq u\quad\text{if and only if}\quad{\sim_u}\subseteq{\sim_t}.
\]

\begin{claim}\label{s025}
There is an $l\in\omega$ such that
every $\sqsubseteq$-chain without repetition must have length less than $l$.
\end{claim}
\begin{proof}[Proof of Claim~\textnormal{\ref{s025}}]
We first prove that for all $u\in\fin(A)^n$,
\begin{equation}\label{s026}
\bigl|\bigl\{t\in\fin(A)^n\bigm|{\sim_t}={\sim_u}\bigr\}\bigr|\leqslant2^{(|B|+2^n)\cdot n}.
\end{equation}
Let $u\in\fin(A)^n$. Let $g$ be the function defined on $\fin(A)^n$ such that for all $t\in\fin(A)^n$,
$g(t)$ is the function on $n$ given by
\[
g(t)(k)=\bigl(t(k)\cap B,\,\bigl\{w\in (A\setminus B)/{\sim_u}\bigm|w\subseteq t(k)\bigr\}\bigr).
\]
Clearly, $\ran(g)\subseteq\bigl(\wp(B)\times\wp((A\setminus B)/{\sim_u})\bigr)^n$.
It is also easy to see that $g{\upharpoonright}\{t\in\fin(A)^n\mid{\sim_t}={\sim_u}\}$ is injective.
Since $|(A\setminus B)/{\sim_u}|\leqslant2^n$, we have
\[
\bigl|\bigl\{t\in\fin(A)^n\bigm|{\sim_t}={\sim_u}\bigr\}\bigr|\leqslant
\bigl|\bigl(\wp(B)\times\wp((A\setminus B)/{\sim_u})\bigr)^n\bigr|\leqslant2^{(|B|+2^n)\cdot n}.
\]

For each $t\in\fin(A)^n$, let $k_t=|(A\setminus B)/{\sim_t}|$.
Clearly, for all $t,u\in\fin(A)^n$ such that $t\sqsubseteq u$, we have $0<k_t\leqslant k_u\leqslant2^n$,
and if $k_t=k_u$ then ${\sim_t}={\sim_u}$. Thus, by \eqref{s026},
every $\sqsubseteq$-chain without repetition must have length less than or equal to $2^{(|B|+2^n)\cdot n}\cdot2^n$.
Now, it suffices to take $l=2^{(|B|+2^n+1)\cdot n}+1$.
\end{proof}

\begin{claim}\label{s027}
For all $u\in\fin(A)^n$ we have $f(u)\sqsubseteq u$.
\end{claim}
\begin{proof}[Proof of Claim~\textnormal{\ref{s027}}]
Assume towards a contradiction that ${\sim_u}\nsubseteq{\sim_{f(u)}}$ for some $u\in\fin(A)^n$.
Let $a,b\in A\setminus B$ be such that $a\sim_ub$ but not $a\sim_{f(u)}b$. Clearly $a\neq b$.
Let $\tau$ be the transposition that swaps $a$ and $b$. Then $\tau(u)=u$ but $\tau(f(u))\neq f(u)$,
contradicting that $B$ is a finite support of $f$.
\end{proof}

We prove \eqref{s024} as follows. Let $t\in\fin(A)^n$. By Claim~\ref{s025},
there is an $l\in\omega$ such that every $\sqsubseteq$-chain without repetition must have length less than $l$.
Let $h$ be a function from $l$ into $\fin(A)^n$, such that $h(0)=t$ and for all $i<l$ if $i+1<l$ then $h(i)=f(h(i+1))$.
Such an $h$ exists since $f$ is surjective. Clearly, for all $i<l$, $f^{(i)}(h(i))=t$.
By Claim~\ref{s027}, $h$ is a $\sqsubseteq$-chain, and since the length of $h$ is $l$,
we can find $i,j<l$ such that $i<j$ and $h(i)=h(j)$. Now, if we take $m=j-i$, then we have $m>0$ and
\[
f^{(m)}(t)=f^{(j-i)}(t)=f^{(j-i)}(f^{(i)}(h(i)))=f^{(j)}(h(j))=t.\qedhere
\]
\end{proof}

Now the following proposition immediately follows from Lemma~\ref{s023} and the Jech--Sochor theorem.

\begin{proposition}\label{s028}
The following statement is consistent with $\mathsf{ZF}$:
there is an infinite set $A$ such that $\fin(A)^n$ is dually Dedekind finite for any $n\in\omega$.
\end{proposition}

\begin{corollary}\label{s029}
The following statement is consistent with $\mathsf{ZF}$:
there exists an infinite cardinal $\mathfrak{a}$ such that
$\fin(\mathfrak{a})^n\nleqslant^\ast[\fin(\mathfrak{a})]^n$ for any $n\geqslant2$.
\end{corollary}
\begin{proof}
Notice that for all infinite sets $A$ and all natural numbers $n\geqslant2$,
there exists a non-injective surjection from $\fin(A)^n$ onto $[\fin(A)]^n$.
Hence, this corollary follows from Proposition~\ref{s028}.
\end{proof}

We conclude this paper with two open problems.

\begin{question}
Is it provable in $\mathsf{ZF}$ that $2^{2^{\fin(\mathfrak{a})}}=2^{2^{\fin(\fin(\mathfrak{a}))}}$ for any infinite cardinal $\mathfrak{a}$?
\end{question}

Notice that Proposition~\ref{s022} shows that $2^{\fin(\mathfrak{a})}=2^{\fin(\fin(\mathfrak{a}))}$
cannot be proved in $\mathsf{ZF}$ for an arbitrary infinite cardinal $\mathfrak{a}$.

\begin{question}\label{s030}
Does $\mathsf{ZF}$ prove that $2^{2^{\mathfrak{a}}}=2^{2^{\mathfrak{a}+1}}$ for any infinite cardinal $\mathfrak{a}$?
\end{question}

Notice that for all Dedekind finite cardinals $\mathfrak{a}$ we have $\mathfrak{a}<\mathfrak{a}+1$,
and for all power Dedekind finite cardinals $\mathfrak{a}$ (i.e., cardinals $\mathfrak{a}$ such that $2^\mathfrak{a}$ is Dedekind finite)
we have $2^\mathfrak{a}<2^{\mathfrak{a}+1}$.

Question~\ref{s030} is asked in \cite{Lauchli1961} (cf.~also \cite[p.~132]{Halbeisen2017}).
Notice that, in \cite{Lauchli1961}, L\"auchli proves in $\mathsf{ZF}$ that for all infinite cardinals $\mathfrak{a}$,
\[
2^{2^{\mathfrak{a}}}=2^{2^{\mathfrak{a}}+1}.
\]

\subsection*{Acknowledgements}
I would like to give thanks to Professor Qi Feng
for his advice and encouragement during the preparation of this paper.


\end{document}